\documentclass[11pt]{amsart}

\usepackage[all]{xy}
\usepackage{color, pstricks}
\usepackage{mathrsfs}
\usepackage{amsmath,amsthm, amssymb, amsgen,amsxtra,amsfonts, amsbsy} 
\usepackage[all]{xy}
\usepackage{verbatim}

\usepackage{dsfont,times}

\begin{document}
%
%
%
\theoremstyle{definition}
\newtheorem{Definition}{Definition}[section]
\newtheorem*{Definitionx}{Definition}
\newtheorem{Convention}{Definition}[section]
\newtheorem{Construction}{Construction}[section]
\newtheorem{Example}[Definition]{Example}
\newtheorem{Examples}[Definition]{Examples}
\newtheorem{Remark}[Definition]{Remark}
\newtheorem*{Remarkx}{Remark}
\newtheorem{Remarks}[Definition]{Remarks}
\newtheorem{Caution}[Definition]{Caution}
\newtheorem{Conjecture}[Definition]{Conjecture}
\newtheorem*{Conjecturex}{Conjecture}
\newtheorem{Question}{Question}
\newtheorem{Questions}[Definition]{Questions}
\newtheorem*{Acknowledgements}{Acknowledgements}
\newtheorem*{Organization}{Organization}
\newtheorem*{Disclaimer}{Disclaimer}
\theoremstyle{plain}
\newtheorem{Theorem}[Definition]{Theorem}
\newtheorem*{Theoremx}{Theorem}
\newtheorem{Proposition}[Definition]{Proposition}
\newtheorem*{Propositionx}{Proposition}
\newtheorem{Lemma}[Definition]{Lemma}
\newtheorem{Corollary}[Definition]{Corollary}
\newtheorem*{Corollaryx}{Corollary}
\newtheorem{Fact}[Definition]{Fact}
\newtheorem{Facts}[Definition]{Facts}
\newtheoremstyle{voiditstyle}{3pt}{3pt}{\itshape}{\parindent}%
{\bfseries}{.}{ }{\thmnote{#3}}%
\theoremstyle{voiditstyle}
\newtheorem*{VoidItalic}{}
\newtheoremstyle{voidromstyle}{3pt}{3pt}{\rm}{\parindent}%
{\bfseries}{.}{ }{\thmnote{#3}}%
\theoremstyle{voidromstyle}
\newtheorem*{VoidRoman}{}

%
\newcommand{\prf}{\par\noindent{\sc Proof.}\quad}
\newcommand{\blowup}{\rule[-3mm]{0mm}{0mm}}
\newcommand{\cal}{\mathcal}
\newcommand{\Aff}{{\mathds{A}}}
\newcommand{\BB}{{\mathds{B}}}
\newcommand{\CC}{{\mathds{C}}}
\newcommand{\EE}{{\mathds{E}}}
\newcommand{\FF}{{\mathds{F}}}
\newcommand{\GG}{{\mathds{G}}}
\newcommand{\HH}{{\mathds{H}}}
\newcommand{\NN}{{\mathds{N}}}
\newcommand{\ZZ}{{\mathds{Z}}}
\newcommand{\PP}{{\mathds{P}}}
\newcommand{\QQ}{{\mathds{Q}}}
\newcommand{\RR}{{\mathds{R}}}
\newcommand{\Liea}{{\mathfrak a}}
\newcommand{\Lieb}{{\mathfrak b}}
\newcommand{\Lieg}{{\mathfrak g}}
\newcommand{\Liem}{{\mathfrak m}}
\newcommand{\ideala}{{\mathfrak a}}
\newcommand{\idealb}{{\mathfrak b}}
\newcommand{\idealg}{{\mathfrak g}}
\newcommand{\idealm}{{\mathfrak m}}
\newcommand{\idealp}{{\mathfrak p}}
\newcommand{\idealq}{{\mathfrak q}}
\newcommand{\idealI}{{\cal I}}
\newcommand{\lin}{\sim}
\newcommand{\num}{\equiv}
\newcommand{\dual}{\ast}
\newcommand{\iso}{\cong}
\newcommand{\homeo}{\approx}
\newcommand{\mm}{{\mathfrak m}}
\newcommand{\pp}{{\mathfrak p}}
\newcommand{\qq}{{\mathfrak q}}
\newcommand{\rr}{{\mathfrak r}}
\newcommand{\pP}{{\mathfrak P}}
\newcommand{\qQ}{{\mathfrak Q}}
\newcommand{\rR}{{\mathfrak R}}
%
%
\newcommand{\OO}{{\cal O}}
\newcommand{\numero}{{n$^{\rm o}\:$}}
\newcommand{\mf}[1]{\mathfrak{#1}}
\newcommand{\mc}[1]{\mathcal{#1}}
\newcommand{\into}{{\hookrightarrow}}
\newcommand{\onto}{{\twoheadrightarrow}}
\newcommand{\Spec}{{\rm Spec}\:}
\newcommand{\BigSpec}{{\rm\bf Spec}\:}
\newcommand{\Spf}{{\rm Spf}\:}
\newcommand{\Proj}{{\rm Proj}\:}
\newcommand{\Pic}{{\rm Pic }}
\newcommand{\Br}{{\rm Br}}
\newcommand{\NS}{{\rm NS}}
\newcommand{\Sym}{{\mathfrak S}}
\newcommand{\Aut}{{\rm Aut}}
\newcommand{\Autp}{{\rm Aut}^p}
\newcommand{\Hom}{{\rm Hom}}
\newcommand{\Ext}{{\rm Ext}}
\newcommand{\ord}{{\rm ord}}
\newcommand{\coker}{{\rm coker}\,}
\newcommand{\divisor}{{\rm div}}
\newcommand{\Def}{{\rm Def}}
\newcommand{\piet}{{\pi_1^{\rm \acute{e}t}}}
\newcommand{\Het}[1]{{H_{\rm \acute{e}t}^{{#1}}}}
\newcommand{\Hcris}[1]{{H_{\rm cris}^{{#1}}}}
\newcommand{\HdR}[1]{{H_{\rm dR}^{{#1}}}}
\newcommand{\hdR}[1]{{h_{\rm dR}^{{#1}}}}
\newcommand{\defin}[1]{{\bf #1}}

\newcommand{\X}{{\mathcal X}}
\newcommand{\Z}{{\mathcal Z}}

\title{On the Birational Nature of Lifting}
\author{Christian Liedtke}
\address{TU M\"unchen, Zentrum Mathematik - M11, Boltzmannstr. 3, D-85748 Garching bei M\"unchen, Germany}
\curraddr{}
\email{liedtke@ma.tum.de}

\author{Matthew Satriano}
\address{Department of Mathematics, University of Michigan, 2074 East Hall, Ann Arbor, MI 48109-1043, USA}
\curraddr{}
\email{satriano@umich.edu}

\date{November 20, 2013}

\begin{abstract}
 Let $X$ and $Y$ be proper birational varieties, say with only rational double points over a perfect field $k$ of positive characteristic.  If $X$ lifts to 
 $W_n(k)$, is it true that $Y$ has the same lifting property? This is true for smooth surfaces, but we show by example that this is false for smooth varieties in higher dimension, and for surfaces with canonical singularities. We also answer a stacky analogue of this question: given a canonical surface $X$ with minimal resolution $Y$ and stacky resolution $\X$, we characterize when liftability of $Y$ is equivalent to that of $\X$.
  
 The main input for our results is a study of how the deformation functor of a canonical surface singularity compares with the deformation functor of its minimal resolution. This extends work of Burns and Wahl to positive characteristic. As a byproduct, we show that Tjurina's vanishing result fails for every canonical surface singularity in every positive characteristic.
\end{abstract}

\maketitle

\section{Introduction}

In 1961, Serre gave a surprising example of a smooth projective variety over a field of positive characteristic which admits no lifting to characteristic 0 \cite{serre}.  The question of whether a variety admits such a lift is oftentimes subtle, and is intimately tied to pathological behavior in positive characteristic.  In this paper, we explore the extent to which liftability is a birational invariant. Since many classification results and constructions in classical algebraic geometry yield singular varieties, and lifting is often easier to establish for these singular models (see, for example \cite{lie}), we will study varieties with mild singularities.

\begin{Question}
 \label{question}
 Let $X$ and $Y$ be proper birational varieties of dimension $d$, say, with at worst rational double points over a perfect field $k$ of positive characteristic.  If $Y$ lifts to $W_n(k)$, is it true that $X$ also lifts to $W_n(k)$?
\end{Question}

Note that this question has two main features: first, we put a bound on the singularities of $X$ and $Y$; second, we ask for unramified lifts, namely lifts to $W_n(k)$ as opposed to extensions of $W_n(k)$.  A bound on the singularities is certainly needed to make Question \ref{question} meaningful.  Indeed, every $d$-dimensional projective variety $X$ is birational, via generic projection, to a hypersurface in $\PP^{d+1}$. This hypersurface may have bad singularities (for example, non-normal), but it always lifts to $W(k)$. On the other hand, $X$ may fail to lift.

Second, recall that there is an important distinction between unramified and ramified lifts of a variety. As is well-known, many fundamental theorems in characteristic 0 fail to hold in positive characteristic: 
global differential forms need not be closed \cite{mumford path} and Kodaira vanishing may fail to hold \cite{raynaud}.  However, if $X$ admits a lift to $W_2(k)$, by a result of Deligne and Illusie \cite{deligne illusie}, these pathologies disappear.  As examples of Lang show \cite{lang}, even if a variety admits a lift to a ramified extension of $W(k)$ with the smallest possible ramification index, namely 2, this is not enough to ensure that global differential forms be closed.  Hence, we restrict attention in Question \ref{question} to the case of unramified lifts.

Question \ref{question} is known to have a positive answer for smooth surfaces. In contrast, we prove the following result for higher dimensional varieties.
\begin{Theorem}
\label{thm:main-higher-dim}
If $d\geq 3$, Question \ref{question} has a negative answer, even if $X$ and $Y$ are smooth. In fact, if $d\geq 5$, there exist
\begin{enumerate}
 \item [(a)] smooth blow-ups of $\PP^d_k$ that do not lift to $W_2(k)$.
 \item [(b)] smooth blow-ups of $\PP^d_k$ that do not lift formally to any ramified extension of $W(k)$.
 \end{enumerate}
\end{Theorem}
Our specific counter-examples in dimensions 3 and 4 are given in Theorem \ref{thm:main2-higher-dim}.  In Theorem \ref{thm:main3-higher-dim} we give further examples of $3$-folds with ordinary double points that lift to $W(k)$, but where small resolutions of singularities do not even lift to $W_2(k)$.

We next turn to the case of surfaces with singularities (see Theorem \ref{thm:main2-surfaces} for the counter-examples).
\begin{Theorem}
\label{thm:main-surfaces}
If $d=2$, Question \ref{question} again has a negative answer; however, if $X$ has at worst rational singularities and $Y$ is smooth, then Question \ref{question} has a positive answer.
\end{Theorem}

Lastly, we explore a variant on Question \ref{question} which constitutes the most subtle part of the paper. If $X$ is a surface with canonical singularities, classically one studies the minimal resolution of singularities
$$
  f\,:\,Y\,\to\,X.
$$
Under a further mild assumption on the singularities of $X$, \cite{cst} shows that there is a smooth stack $\X$ with coarse space $X$ whose stacky structure lies over the singular points of $X$. That is, we have a stacky resolution
$$
\pi\,:\,\X\,\to\,X.
$$
The interplay between the birational geometry of $Y$ and $\X$ in characteristic 0 has been the source of many interesting questions, for example, the McKay correspondence \cite{mckay1,mckay2}. Here we ask another question concerning the birational geometry of $Y$ and $\X$, namely the stacky version of Question \ref{question}: is liftability of $\X$ equivalent to that of $Y$?

Since Question \ref{question} has an affirmative answer for smooth surfaces, one might expect that liftability of the smooth stacky surfaces $\X$ and $Y$ is equivalent. We show that this is the case precisely when $X$ does not have wild $A_n$-singularities, that is, $A_n$-singularities with $p$ dividing $n+1$.

\begin{Theorem}
\label{thm:main-stacks}
Let $X$ be a proper surface with canonical singularities that are linearly reductive quotient singularities (see Definition \ref{def:linearly reductive}).
 \begin{enumerate}
  \item \label{stacks:X==>Y} If $\X$ lifts to $W_n(k)$, then $Y$ does as well.
  \item \label{stacks:Y==>X} If $X$ has canonical singularities and no wild $A_n$-singularities, then liftability of $Y$ to $W_2(k)$ implies that of $\X$.
   \item \label{stacks:negative} In characteristic $2$, there is a singular K3 surface $X$ with only (canonical wild) $A_1$-singularities such that $X$ and $Y$ lift formally to $W(k)$, but $\X$ does not lift to $W_2(k)$.
  \end{enumerate}
\end{Theorem}

The main input for Theorem \ref{thm:main-stacks} is a study of the relationship between the deformation functor of an isolated canonical singularity and the deformation functor of its minimal resolution. This analysis, which we carry out in \S\ref{sec:duval}, extends results of Burns and Wahl \cite{bw} to positive characteristic, and supplements the work of Wahl \cite{wahl}. 
We show that for canonical singularities that are linearly reductive quotient singularities but not wild $A_n$-singularities, many results from \cite{bw} still hold true in positive characteristic. 
On the other hand, we show in Remark \ref{tjurina fails} that Tjurina's vanishing result \cite{tjurina} fails for every canonical singularity in every positive characteristic.

We conclude the introduction by mentioning that our above results also answer the following variant on Question \ref{question}.

\begin{Question}
 \label{question2}
 Let $X$ and $Y$ be proper birational varieties with at worst rational double points over a perfect field $k$ of positive characteristic.  If $Y$ lifts to $W(k)$, does $X$ lifts to an \emph{extension} of $W(k)$?
\end{Question}

Although Theorem \ref{thm:main-surfaces} shows that Question \ref{question} has a negative answer for surfaces, Artin's result \cite[Theorem 3]{artin simult res} show that Question \ref{question2} has a \emph{positive} answer for surfaces.  In contrast, the examples we produce in Theorems \ref{thm:main-higher-dim}(b) and \ref{thm:main3-higher-dim} show that Question \ref{question2} has a negative answer in higher dimension.

\begin{Organization}
In Section \ref{sec:higher dimension} we start with a couple of general lifting results and then construct counter-examples to Question \ref{question} in dimension $\geq3$, thereby establishing 
Theorem \ref{thm:main-higher-dim}.
In Section \ref{sec:dimension two} we turn to surfaces and establish the results sketched in Theorem \ref{thm:main-surfaces}.
We begin Section \ref{sec:duval} by recalling the definition of linearly reductive quotient singularities, and giving a complete description of which canonical singularities are of this form. 
We then study the deformation functors of these singularities and obtain counter-examples to Tjurina vanishing.
Finally, in Section \ref{sec:stackyres}, we compare minimal with stacky resolutions of canonical and linearly reductive quotient singularities of surfaces, which leads to a proof
of Theorem \ref{thm:main-stacks}.
\end{Organization}

\begin{Acknowledgements}
 We would like to thank Dan Abramovich, Bhargav Bhatt, Brian Conrad, 
 S\l awomir Cynk, Torsten Ekedahl, Anton Geraschenko, Jesse Kass, Holger Partsch, David Rydh,
 Felix Sch\"uller, Matthias Sch\"utt, and the referee for helpful conversations.
 The first named author was supported by DFG grant LI 1906/1-2  and Transregio SFB 45 
 and thanks the departments of mathematics 
 at Stanford university and Bonn University for kind hospitality.  
 The second named author was supported by NSF grant DMS-0943832 and an NSF postdoctoral fellowship (DMS-1103788).
\end{Acknowledgements}

\section*{Notation and Conventions}

%
Unless otherwise mentioned, all algebraic stacks are assumed to be locally of finite presentation with finite diagonal, so that by Keel--Mori \cite{km}, they have coarse spaces.

For a scheme $X$ over $k$, we let $\Theta_X:=\mc{H}om(\Omega^1_{X/k},\OO_X)$.


\section{Counter-examples in higher dimension}
\label{sec:higher dimension}

In this section, we first recall in \S\ref{subsec:genres} some general results concerning liftings and blow-downs, mostly following directly from \cite{bw}.
Then, in \S\ref{subsec:higher-dim-counter-ex}, we give examples of smooth, projective and birational varieties of dimension at least $3$ with different
lifting behaviors.
More precisely, we prove
Theorems \ref{thm:mainbody-higher-dim}, \ref{thm:main2-higher-dim}, and \ref{thm:main3-higher-dim} which give refined versions of Theorem \ref{thm:main-higher-dim}.

\subsection{General lifting results}
\label{subsec:genres}
Throughout this subsection, let $A$ be a complete Noetherian local ring with perfect residue field $k$. We begin by recalling a result of Burns and Wahl \cite[Proposition 2.3]{bw} which shows that certain deformations can be blown-down.
In the following form, the result is due to Cynk and van~Straten \cite[Theorem 3.1]{CvS}.

\begin{Proposition}[Burns--Wahl, Cynk--van~Straten]
 \label{burns-wahl}
 Let $X$ and $Y$ be schemes over $k$. 
 Let $f:Y\to X$ be a morphism such that $Rf_\ast\OO_Y=\OO_X$.
 If $Y$ formally lifts to $A$, then $X$ does as well.  Explicitly, if $Y'$ is a formal 
 lift of $Y$ to $A$, then we 
 may view $\OO_{Y'}$ as a sheaf on the topological space $Y$; the topological space of $X$ 
 endowed with the sheaf $f_*\OO_{Y'}$ is a lift of $X$ to $A$.
\end{Proposition}

We continue with a simple lifting result, which shows that in certain cases,
Question \ref{question} has an affirmative answer.
On the other hand, the counter-examples 
in \S\ref{subsec:higher-dim-counter-ex} below will show that one should neither expect
the converse lifting implications to hold nor to hope for more general lifting results 
in dimension at least $3$. 

\begin{Proposition}
 \label{higher lifting prop}
 Let $f:Y\to X$ be a birational morphism between two smooth proper varieties over $k$.
 \begin{enumerate}
  \item If $Y$ lifts formally to $A$, then $X$ does as well.
  \item If $f$ is the blow-up of a closed point and $X$ lifts formally to $A$, then $Y$ lifts to $A$.
  \item If $f$ is the blow-up of a smooth subvariety $Z\subset X$ of codimension at least $2$,
    and if $Y$ lifts formally to $A$, then so do $Z$ and $X$. Moreover, there also exists a formal 
    lift of $Z$ as a subvariety of $X$.
 \end{enumerate}
\end{Proposition}

\begin{proof}
By \cite[Corollary 3.2.4]{rulling}, we have $Rf_\ast\OO_Y=\OO_X$, and so (1) follows from Proposition \ref{burns-wahl}.

To prove (2), let $X'$ be a lift of $X$ to $\Spf A$. Since $X'$ is smooth over $\Spf A$ and $k$ is perfect, 
there exists a local and \'etale $A$-algebra $B$ together
with a morphism $\sigma:\Spf B\to X'$ that specializes to the closed point of the blow-up $f$,
see \cite[Proposition 2.2.14]{BLR}.
Then, the blow-up of $X'$ in $\sigma(\Spf B)$ is a formal lift of $Y$ to $A$.

Lastly, we prove (3).  Let $E$ be the 
exceptional divisor of $f$, and let $Y'$ be a formal lift of $Y$ to $A$. 
The normal bundle $N_{E/Y}$ restricts to
$\OO(-1)$ on every fiber of the projective bundle $g:E\to Z$, and 
then, the Grothendieck--Leray spectral sequence of $g$ implies
$H^1(N_{E/Y})=0$, see, for example, \cite[Examples 3.14.13(iv)]{sernesi}. 
Since $E$ and $Y$ are smooth, the obstruction to deforming $E\subset Y$ is contained in $H^1(N_{E/Y})$,
and we conclude that $E$ lifts to a closed subscheme $E'\subset Y'$. 
By Proposition \ref{burns-wahl}, 
we obtain a formal lift $X'$ of $X$ and 
a lift of $Z$ to a closed subscheme of $X'$.
\end{proof}

\subsection{Counter-examples}
\label{subsec:higher-dim-counter-ex}
We begin this subsection with the counter-examples which were announced as Theorem \ref{thm:main-higher-dim}
in dimension $d\geq 5$.

\begin{Theorem}
\label{thm:mainbody-higher-dim}
 Let $k$ be an algebraically closed field of positive characteristic and let $d\geq5$.  Then there exist blow-ups in smooth centers 
 \begin{enumerate}
  \item [(a)] $f_1:Y_1\to\PP_k^d$ such that
     $Y_1$ does not lift to $W_2(k)$, and
  \item [(b)] $f_2:Y_2\to\PP_k^d$ such that 
     $Y_2$ does not lift formally to any ramified extension of $W(k)$.
 \end{enumerate}
 On the other hand, $\PP_k^d$ lifts projectively to $W(k)$.
\end{Theorem}

\begin{proof}
Let $S_1$ be a smooth projective surface over $k$ that does not 
lift to $W_2(k)$.
For example, we could choose $S_1$ to be a characteristic $p$
counter-example to Kodaira vanishing from \cite[\S2]{raynaud}, which cannot
lift to $W_2(k)$ by \cite[Corollaire 2.8]{deligne illusie}. 
Since every smooth and projective surface over $k$ 
can be embedded into $\PP_k^5$, 
we may assume $S_1\subset\PP_k^5\subseteq\PP_k^d$.
If $f_1:Y_1\to\PP_k^d$ is the blow-up in $S_1$, then
$Y_1$ does not lift to $W_2(k)$ by Proposition \ref{higher lifting prop}(3).

Next, let $S_2$ be a smooth projective surface over $k$ that does not lift
projectively to any ramified extension of $W(k)$. 
For example, we could choose $S_2$ to be a characteristic $p$ counter-example 
to the Bogomolov--Miyaoka--Yau inequality from \cite[3.5J]{hirzebruch} or \cite{easton}:
then, since $K^2$ and $\chi(\OO)$ are invariant under flat deformations, 
a hypothetical projective lift of $S_2$ to a possibly ramified extension of $W(k)$ would contradict
the Bogomolov--Miyaoka--Yau inequality in characteristic zero.
As before, we choose embeddings $S_2\subseteq\PP_k^5\subseteq\PP_k^d$
and let $f_2:Y_2\to\PP_k^d$ be the blow-up along $S_2$.

Suppose $Y_2$ has a formal lift $Y'_2$ to a possibly ramified extension $R$ of $W(k)$. 
Since $Y_2$ is a smooth and rational variety, we have $H^2(\OO_{Y_2})=0$. Since $H^2(\OO_{Y_2})$ is the 
obstruction space to deforming invertible sheaves, every invertible sheaf of $Y_2$ lifts to $Y'_2$. 
Therefore, lifting an ample invertible sheaf to $Y_2'$, we conclude that $Y'_2$ is algebraizable and projective by 
Grothendieck's existence theorem \cite[Theorem 8.4.10]{GET}. 
By Proposition \ref{higher lifting prop}(3), we obtain a projective lift of $S_2$ to $R$, which is a contradiction.
\end{proof}

Next, we give lower dimensional counter-examples, whose constructions are
inspired by Raynaud's construction of characteristic $p$ counter-examples to Kodaira vanishing \cite[\S2]{raynaud}.

\begin{Theorem}
\label{thm:main2-higher-dim}
 For every algebraically closed field $k$ of positive characteristic and integer $d\geq3$, there exists 
 \begin{enumerate}
  \item a smooth ruled $d$-dimensional variety $X$ over $k$ that lifts projectively to $W(k)$, and
  \item a blow-up $f:Y\to X$ in a smooth curve such that $Y$ does not lift to $W_2(k)$.
 \end{enumerate}
\end{Theorem}

\begin{proof}
By \cite[\S2]{raynaud}, there exists a projective smooth curve $C$ of genus at least $2$ over $k$, 
a locally free sheaf ${\mathcal E}$ of rank $2$ on $C$, 
and a closed subscheme $D$ of the surface $\PP({\mathcal E})$ satisfying the following: 
$D$ is a smooth curve, and the composite $D\to\PP({\mathcal E})\to C$ induces the $k$-linear
Frobenius morphism $D\to D^{(p)}\iso C$. 
Let $X$ be the smooth ruled $d$-fold $\pi:\PP({\mathcal E}\oplus \OO_C^{d-2})\to C$. 
The projection ${\mathcal E}\oplus\OO_C^{d-2}\to {\mathcal E}$ onto the first summand
induces an embedding of $\PP({\mathcal E})$ into $X$ over $C$. We let $f:Y\to X$ be the blow-up along $D$.

We first show that $X$ lifts projectively to $W(k)$. 
Since $C$ is a projective smooth curve, it lifts projectively to some $\widetilde{C}$ over $W(k)$. 
We have $H^2(C,{\mathcal E}nd({\mathcal E}))=0$ for dimensional reasons, and thus, ${\cal E}$ lifts to
some $\widetilde{{\cal E}}$ on 
$\widetilde{C}$,  see \cite[Theorem 8.5.3]{GET}. 
In particular, $\PP(\widetilde{{\cal E}})\to\widetilde{C}$ defines a projective lift of $X$ to $W(k)$.

Next, we show that $Y$ does not lift to $W_2(k)$. If it lifts, then by Proposition \ref{higher lifting prop}(3), we obtain a lift $X'$ of $X$ to $W_2(k)$ and a lift $D'\subset X'$ of $D\subset X$. 
Since $\pi:X\to C$ is a projective bundle, we have $R\pi_\ast\OO_X=\OO_C$, and thus, $X'$ induces a lift $C'$ of $C$ to $W_2(k)$ by
Proposition \ref{burns-wahl}.
The composite $D'\to X'\to C'$ is then a lift of Frobenius to $W_2(k)$, which is impossible by \cite[Lemma I.5.4]{raynaud2}.
This contradiction shows that $Y$ does not lift to $W_2(k)$.
\end{proof}

Finally, we show that there exist $3$-folds with ordinary double points
that lift to $W(k)$, but where small resolutions of singularities do not even lift to $W_2(k)$. 
We recall that the ordinary $3$-dimensional double point is defined to be
$$
 k[[x,y,z,w]]/(xy-zw).
$$
In every characteristic, this singularity is normal, Gorenstein, and blowing up the singular point we obtain a resolution with exceptional 
locus $\PP^1\times\PP^1$.
Contracting one of the two factors of $\PP^1\times\PP^1$, we obtain a small
resolution with exceptional locus $\PP^1$ and normal bundle
$\OO_{\PP^1}(-1)\oplus\OO_{\PP^1}(-1)$, see \cite[\S4]{CvS}.
By definition, the induced birational rational map between these two small resolutions is the
Atiyah flop.
 

\begin{Proposition}
 Let $X$ be a $3$-dimensional variety with one 
 singular point that is an ordinary double point. Let $Y_i\to X$ for $i=1,2$ be the two small resolutions of the singularities described above. Then $Y_1$ lifts formally to $A$ if and only if $Y_2$ does.
\end{Proposition}

\begin{proof}
Let $f_i:Y_i\to X$, $i=1,2$ be the contraction morphisms of the respective
flopping curves.
By assumption, $X$ has an ordinary double point.
A formal lift $Y'_1$ of $Y_1$ to $A$  induces a formal
lift $X'$ of $X$ to $A$ by Proposition \ref{burns-wahl} and \cite[Theorem 4.1]{CvS}.
The induced lift to $A$ of the ordinary double point of $X$ 
is determined in \cite[page 237]{CvS} and in particular,
$Y_1'\to X'$ is the blow-up in a singular section.
By blowing up the other singular section that comes with this
particular lift of the ordinary double point to $A$
(see \cite[p. 237]{CvS} and \cite[Proposition 4.2]{CvS}), we obtain a lift of
$Y_2$ to $A$.
\end{proof}

Whereas the lifting behavior does not change under Atiyah flops, it 
may change under small resolutions of $3$-fold ordinary double points, 
as the following examples show.

\begin{Theorem}[Cynk--van Straten, Schoen, $+\varepsilon$]
\label{thm:main3-higher-dim}
 For every prime
 $$
   p\,\in\,\{3,5,7,11,17,29,41,73,251,919,9001\}
 $$
 there exists a projective Calabi--Yau $3$-fold $X$ over $k=\FF_p$ with only ordinary double points as singularities with the following properties:
 \begin{enumerate}
  \item $X$ lifts projectively to $W(k)$,
  \item there exist small resolutions of singularities $Y\to X$ in the category of algebraic spaces, but none of them lifts to $W_2(k)$ or 
  formally to a ramified extension of $W(k)$,
  \item there exist projective resolutions of singularities $Z\to X$ that neither lift to $W_2(k)$ nor formally to a ramified extension of $W(k)$.
 \end{enumerate}
\end{Theorem}

\begin{proof}
Our examples arise as fiber products of rational elliptic surfaces and their desingularizations. 
Rational elliptic surfaces that are semi-stable as elliptic fibrations over $\PP^1$ with precisely $4$ singular fibers were classified in \cite{beauville}, and we refer to \cite[\S4]{schoen} for a characteristic-free classification.
As shown in \cite[\S6.2]{CvS}, we may 
find for all $p$ as in the statement of the theorem
two rational elliptic surfaces ${\cal S}_i\to\PP^1$, $i=1,2$ over $W(k)$,
${\rm char}(k)=p$, whose elliptic fibrations are semi-stable, and
whose $4$ singular fibers 
lie over $\{0,1,\lambda,\infty\}$ and $\{0,1,\mu,\infty\}$,
respectively.
It is further shown that there are examples where the fibers over $\lambda$ and $\mu$
are of type $I_1$, that $\lambda\neq\mu$, and that $\lambda\equiv\mu\mod p$.
Moreover, using the explicit equations of \cite[Table 1]{schoen}, 
we may assume  $\lambda\not\equiv\mu\mod p^2$.
The fiber product 
${\cal X}:={\cal S}_1\times_{\PP^1}{\cal S}_2$ is 
projective of relative dimension $3$ over $W(k)$.
As explained in \cite[\S6.2]{CvS}, the singularities of the generic fiber ${\cal X}_\eta$ are 
ordinary double points lying over $\{0,1,\infty\}\subset\PP^1$,
whereas the special fiber $X:={\cal X}_k$ has an extra double point lying
over $\lambda\mod p$, which, by assumption, is also equal to $\mu\mod p$.

By blowing up the reduced singular locus of $X$, we obtain a projective 
resolution of singularities $Z\to X$.
Since $\Het{3}(Z,\QQ_{\ell})=0$ for every prime $\ell\neq p$
by \cite[Cororollary 3.2]{schoen}, it follows 
from \cite[Proposition 11.1]{schoen} that $Z$ does not admit a lift even
to a ramified extension of $W(k)$.

By \cite[Lemma 3.1]{schoen2} or \cite[\S6.2]{CvS},
there exists a small resolution $\overline{{\cal X}}\to {\cal X}$
of the $3$ double points lying over $\{0,1,\infty\}$ in the category of
algebraic spaces.
The reduction $\overline{X}$ of $\overline{{\cal X}}$ modulo $p$ 
has precisely one double point and thus, is
a partial resolution of singularities of $X$.
Let $Y\to \overline{X}$ be a small resolution of the remaining double point,
still in the category of algebraic spaces.
Since $Y$ is rigid \cite[Proposition 6.3]{CvS}, and $\overline{X}$ has precisely
one double point, it follows from \cite[Remark 4.5]{CvS} that $\overline{X}$ is also rigid.
Thus, $Y$ does not lift to a ramified extension of $W(k)$
by \cite[Theorem 4.3]{CvS}
(although this result is stated for schemes, it also holds
for algebraic spaces, see the discussion on \cite[page 242]{CvS}).

We claim that neither $Z$ nor $Y$ lifts to 
$W_2(k)$:
since there exists a dominant birational morphism $g:Z\to Y$ that satisfies $Rg_\ast\OO_Z=\OO_Y$ by \cite[Corollary 3.2.4]{rulling}, 
it suffices to show that $Y$ does not lift to $W_2(k)$ by Proposition \ref{burns-wahl}.
Thus, assume to the contrary that $Y$ lifts to some $Y'$ over$W_2(k)$. 
This lift blows down to a lift $\overline{X}'$ of $\overline{X}$ by \cite[Theorem 4.1]{CvS} and Proposition \ref{burns-wahl}. 
By \cite[page 237]{CvS}, the induced lift of the double point to $W_2(k)$ is analytically equivalent to
$$
W_2(k)[[x,y,z,w]]/(xy-zw).
$$
On the other hand, the elliptic fibration ${\cal S}_1\to\PP^1$, is given
formally locally over $\{\lambda\}\in\PP^1$ by
$$
\Spf W(k)[[x,y,t]]/(y^2-x^3-x^2-t)\,\to\,\Spf W(k)[[t]],
$$
and similarly for ${\cal S}_2\to\PP^1$ over $\{\mu\}$.
Thus, their fiber product $\overline{{\cal X}}\to\PP^1$ is locally formally
over $\{\lambda\}$ given by
$$
\Spf W(k)[[x,y,u,w,t ]]/( y^2-x^3-x^2-t, w^2-u^3-u^2-(t+\lambda-\mu) )\to\Spf W(k)[[t]].
$$
After eliminating $t$, we see that the ordinary double point of $\overline{X}$ 
deforms in this particular lift to $W(k)$ as
$$
\Spf W(k)[[x,y,u,w]]/( (y^2-x^2-x^3)-(w^2-u^2-u^3) + (\lambda-\mu)).
$$
By rigidity of $\overline{X}$, the lift $\overline{X}'$ is isomorphic 
to $\overline{{\cal X}}\otimes_{W(k)}W_2(k)$.
Since $\lambda\equiv\mu\mod p$ and $\lambda\not\equiv\mu\mod p^2$, we see that
the induced lift of the double point  of $\overline{X}$ to $W_2(k)$
is analytically equivalent to
$$
W_2(k)[[x,y,z,w]]/(xy-zw-p),
$$
a contradiction (see \cite[Remark 5.3]{CvS} for a similar argument).
Thus, $Y$ does not lift to $W_2(k)$.
\end{proof}

\section{On the birational nature of lifting for surfaces}
\label{sec:dimension two}

In this section, we show that smooth and birational surfaces have the same lifting behavior, as announced in
Theorem \ref{thm:main-surfaces}.  We begin with Propositions \ref{smooth surfaces} and \ref{best answer}, which give the well-known positive results of the theorem.

\begin{Proposition}
 \label{smooth surfaces}
 Let $A$ be a complete Noetherian local ring with perfect residue field $k$. 
 Let $X$ and $Y$ be smooth proper birational surfaces over $k$. 
 Then $X$ lifts formally to $A$ if and only if $Y$ does.
\end{Proposition}

\begin{proof}
From the structure result of birational maps, it follows that there exists a
smooth surface $Z$ over $k$, and proper birational morphisms $Z\to X$ and $Z\to Y$.  
Moreover, these proper birational morphisms can be factored into sequences of blow-ups at closed points. 
Thus, it suffices to treat the case where $f:Y\to X$ is the blow-up at a closed point. In this situation, if $Y$ lifts to $A$, then so does $X$ by Proposition \ref{burns-wahl}. Conversely, if $X$ lifts to $A$, then so does $Y$ by Proposition \ref{higher lifting prop}(2).
\end{proof}

\begin{Proposition}
 \label{best answer}
 Let $A$ be a complete Noetherian local ring with residue field $k$. Let $X$ and $Y$ be proper birational surfaces over $k$ with $Y$ smooth and $X$ at worst rational singularities. If $Y$ lifts formally to $A$, then $X$ does as well.
\end{Proposition}

\begin{proof}
Let $f:Z\to X$ be a resolution of singularities. By assumption, it satisfies $Rf_\ast\OO_Z=\OO_X$. Since $Y$ lifts to $A$, so does $Z$ by Proposition \ref{smooth surfaces}. Hence, $X$ lifts to $A$ by Proposition \ref{burns-wahl}.
\end{proof}

\begin{Remark}
 As stated above, Propositions \ref{smooth surfaces} and \ref{best answer} are known.  For a deformation theoretic proof of the former, see \cite[Proposition 1.2.2]{suh}.
\end{Remark}

Lastly, we construct the counter-examples of Theorem \ref{thm:main-surfaces}. In fact, every sufficiently general purely inseparable cover of degree $p$ of $\PP^2$ gives an example, and we thank Torsten Ekedahl for pointing this out to us.

\begin{Theorem}
\label{thm:main2-surfaces}
 For every algebraically closed field $k$ of characteristic $p\geq7$, there exists 
 \begin{enumerate}
  \item a surface $X$ with canonical singularities that lifts projectively to $W(k)$, whereas
  \item no smooth model of $X$ lifts to $W_2(k)$.
 \end{enumerate}
\end{Theorem}

\begin{proof}
Let $n\geq1$ be an integer and let $s$ be a generic section of $H^0(\PP_k^2,{\cal L}^{\otimes p})$, where ${\cal L}:=\OO_{\PP^2}(n)$. Then, the obvious multiplication ${\cal L}^{\otimes(-i)}\oplus{\cal L}^{\otimes(-j)}\to{\cal L}^{\otimes(-i-j)}$ and multiplication by $s:{\cal L}^{\otimes(-i-j)}\to{\cal L}^{\otimes(-i-j+p)}$ turn ${\cal A}:=\bigoplus_{i=0}^{p-1} {\cal L}^{\otimes(-i)}$ into an $\OO_{\PP^2_k}$-algebra. We let $X=\BigSpec{\cal A}$ and $f:X\to \PP^2_k$ be the structure morphism. Since $s$ is generic, $X$ is a surface with at worst canonical singularities of type $A_{p-1}$, see \cite[page 105]{Gauge} or \cite[Theorem 3.4]{lie}. Via lifting ${\cal L}$ and $s$ to $\PP^2_{W(k)}$, we obtain a lift of the whole cover $f:X\to\PP^2_k$ to $W(k)$. In particular, $X$ admits a projective lift to $W(k)$.

Next, let $\widetilde{X}\to X$ be a resolution of singularities.
By \cite[Chapter VI.xiv]{Gauge}, we have
$$
h^0(\widetilde{X},\Omega_{\widetilde{X}}^1)\,\geq\,
\frac{p\left[(p-1)(p-2)-3p\right]}{6}n^2\,-\,\frac{p(p+5)}{4}n\,-\,\frac{3p^2-7p}{4}\,+\,1.
$$
Since we assumed $p\geq7$, there will be non-zero $1$-forms on $\widetilde{X}$ if $n$ is sufficiently large. By \cite[Chapter VI.xiii]{Gauge}, none of these $1$-forms is $d$-closed. 
In particular, there is a non-trivial differential in 
the Fr\"olicher spectral sequence from Hodge-to-de Rham cohomology 
on the $E_1$ page.
However, if  $\widetilde{X}$ were to lift to $W_2(k)$
then its Fr\"olicher spectral sequence would degenerate at
$E_1$ by \cite[Corollaire 2.4]{deligne illusie}, a contradiction.
\end{proof}

\begin{Remark}
While no desingularization of $X$ lifts to $W_2(k)$, it follows from Artin's result
\cite[Theorem 3]{artin simult res} 
that every smooth model lifts formally to a ramified extension of $W(k)$.
\end{Remark}

\section{Canonical Surface Singularities}
\label{sec:duval}

In this section we study the relationship between the deformation functor of an isolated 
canonical surface singularity 
and the deformation functor of its minimal resolution.  
This generalizes many of the results of Burns--Wahl \cite[\S1--2]{bw} to positive characteristic, 
and supplements the analysis of Wahl \cite{wahl}.

In \S\ref{subsec:linred}, we introduce the definition of linearly reductive quotient singularities and give a complete characterization of canonical surface singularities that are of this form. In \S\ref{subsec:defthy} we turn to the study of deformation functors.


\subsection{Linearly reductive quotient singularities}
\label{subsec:linred}

Over the complex numbers and in dimension $2$, rational double points are also known as
canonical singularities, Du~Val singularities, 
ADE singularities, or Kleinian singularities, and they coincide with the class of 
rational Gorenstein singularities. 
Moreover, these singularities are precisely those which are analytically 
quotients by finite subgroups of ${\rm SL}_2(\CC)$,  and we refer to \cite{durfee} for an overview.

In positive characteristic, it is no longer true that every canonical surface singularity is a quotient of a smooth surface by a finite group, see Remark \ref{linearly reductive remark} below. However, we show in Proposition \ref{canonical linearly reductive} that most canonical surface singularities are examples of the following type of singularity:

\begin{Definition}
 \label{def:linearly reductive}
  A scheme over a field $k$ has \emph{linearly reductive quotient singularities} (resp.~ \emph{tame quotient singularities}) if it is \'etale locally isomorphic to the quotient of a smooth $k$-scheme by a finite linearly reductive group scheme (resp.~ finite \'etale group scheme of order prime to the characteristic of $k$).
\end{Definition}

Note that tame quotient singularities are examples of linearly reductive quotient singularities. 
Although these two classes of singularities differ in positive characteristic, they agree in characteristic 0
since finite linearly reductive group schemes in characteristic 0 are all locally constant.

We recall from \cite[Theorem 2.7]{artin numerical} that in any characteristic,
the dual resolution graph for the minimal resolution of a canonical surface
singularity over an algebraically closed field is a Dynkin diagram of type $A$, $D$, or $E$.

\begin{Proposition}
 \label{canonical linearly reductive}
 Let $k$ be an algebraically closed  field of characteristic $p$. The following table summarizes when canonical
 surface singularities over $k$ are linearly reductive quotient singularities (resp. ~tame quotient singularities):
 $$
   \begin{array}{ccc}
    &\mbox{ \underline{linearly reductive quotient singularity} }&\mbox{ \underline{tame quotient singularity} }\\
    A_{n-1} & \textrm{every\ } p & p\!\!\not|n\\
    D_{n+2} & p\geq3 & p\geq3,\, p\!\!\not|n\\
    E_6 & p\geq5 & p\geq5\\
    E_7 & p\geq5 & p\geq5\\
    E_8 & p\geq7 & p\geq7\\
   \end{array}
 $$
In particular, if $p\geq 7$ every canonical surface singularity over $k$ is a linearly reductive quotient singularity.
\end{Proposition}

\begin{Caution}
 In Artin's terminology from \cite[\S2]{artin}, {\em tame} means that the local 
 fundamental group is of order prime-to-$p$.
 In particular, a canonical singularity that
 is a tame quotient singularity, 
 is also tame in Artin's sense but the converse is not true.
 For example, an $A_{n-1}$-singularity is always tame in the sense
 of Artin, but it is a tame quotient singularity if and only if
 $p$ does not divide $n$.
\end{Caution}

\begin{proof}
By Artin's approximation results \cite[Theorem (3.10)]{artin approx}, it suffices
to show that a canonical surface singularity is analytically isomorphic to the 
quotient of $\Aff_k^2$ by a finite linearly reductive group scheme, or 
a finite flat group scheme of order prime to $p$, respectively.

We begin with our assertions on $A_{n-1}$-singularities.  
Such singularities are analytically isomorphic to $k[[u,v,w]]/(uv-w^n)$, by \cite[(2.3)]{artin}.
We can realize this as the complete local ring at the singular point of the quotient $\Aff_k^2/\mu_n$, where the action
$$
  \mu_n\,\times\,\Aff_k^2\,\to\,\Aff_k^2
$$
of $\mu_n$ on $\Aff_k^2$ is given by the map
$$
 \begin{array}{ccc}
   k[x,y] & \to & k[x,y]\otimes_k k[t]/(t^n-1)\\
   x &\mapsto&x\otimes t\\
   y &\mapsto&y\otimes t^{n-1}
  \end{array}
$$
Alternatively, the action can be described as follows: for any $k$-scheme $T$, the action of $\zeta\in\mu_n(T)$ is given on $T$-valued points by sending $f(x,y)\in\Aff_T^2(T)$ to $f(\zeta x,\zeta^{-1} y)$. 
This proves our assertion for $A_{n-1}$-singularities, as the group scheme $\mu_n$ is linearly reductive for all $p$ 
and it is of order prime to $p$ precisely when $p$ does not divide $n$.

We now turn to $D_{n+2}$-singularities. For $n\geq2$ and $p\geq3$, we
consider the closed subscheme ${\rm BD}_{n}$ 
of ${\rm SL}_2$ defined by the surjection
$$
 \begin{array}{ccc}
   k[a_{11},a_{12},a_{21},a_{22}]/(a_{11}a_{22}-a_{12}a_{21}-1)
   &\to& k[a,b]/( (a^{2n}-1)(b^{2n}-1), ab)\\
   a_{11} &\mapsto&a\\
   a_{12} &\mapsto&b\\
   a_{21} &\mapsto&-b^{2n-1}\\
   a_{22} &\mapsto&a^{2n-1}
 \end{array}
$$
The Hopf algebra structure on $k[{\rm SL}_2]$ induces a Hopf algebra structure on ${\rm BD}_{n}$.  
Moreover, we have a short exact sequence 
$$
 1\,\to\,\mu_{2n} \,\to\, {\rm BD}_{n} \,\to\, \mu_2 \,\to\,1,
$$
where the inclusion of $\mu_{2n}$ into ${\rm BD}_{n}$ is defined by 
$$
 \begin{array}{cccc}
  k[a,b]/( (a^{2n}-1)(b^{2n}-1), ab) &\to&k[z]/(z^{2n}-1)\\
  a&\mapsto& z \\
  b&\mapsto& 0 &.
 \end{array}
$$
Nagata's theorem \cite[Theorem 2]{nagata} therefore shows that ${\rm BD}_{n}$ is linearly reductive.  

If $p$ does not divide $n$, then ${\rm BD}_{n}$ is the constant
group scheme associated to the binary dihedral group of order $4n$, and corresponds
to the subgroup of ${\rm SL}_2$ generated by the matrices
$$
 \left(
 \begin{array}{cc}
   \zeta & 0\\ 0& \zeta^{-1}
 \end{array}
 \right)
\mbox{ \quad and \quad } 
 \left(
 \begin{array}{cc}
   0&1\\-1&0
 \end{array}
\right)
$$
where $\zeta$ is a primitive $2n$.th root of unity.

The standard action of ${\rm SL}_2$ on $k[x,y]$ induces an action of ${\rm BD}_n$, and a straightforward computation of invariants reveals
\begin{eqnarray*}
 k[x,y]^{{\rm BD}_n} &\cong&
  k[x^{2n}+y^{2n},(xy)^2,xy(x^{2n}-y^{2n})] \\
 &\cong&
  k[u,v,w]/(w^2-u^2v-4v^{n+1})\,. 
\end{eqnarray*}
Thus, $\Aff_k^2/{\rm BD}_{n}$ is analytically isomorphic to
the singularity of type $D_{n+2}$ in every characteristic $p\geq3$,
and our assertions on $D_{n+2}$-singularities follow.

Finally, the claims for $E_6$, $E_7$ and $E_8$ follow from the classification and the local fundamental groups 
of these singularities in \cite[\S5]{artin}.
\end{proof}

\begin{Remark}
  \label{linearly reductive remark}
  By a result of Mumford \cite{mumford topology}, a two-dimensional, normal, complex analytic germ is smooth if and only
  if its local fundamental group is trivial.
  This is wrong in positive characteristic, but a version using the local Nori fundamental group scheme was given in 
  \cite[Section 4]{esnault viehweg}.
  In any case, for a linearly reductive quotient singularity in characteristic $p$, the local Nori fundamental group scheme 
  is linearly reductive, which implies that its local \'etale fundamental group cannot have subquotients of order $p$.
  
  Now, for the remaining combinations $(\Gamma, p)$ of simply laced Dynkin diagram $\Gamma$
  and prime $p\leq5$ not in the table of Proposition \ref{canonical linearly reductive}, there exist canonical surface singularities
  of type $\Gamma$ in characteristic $p$, whose local fundamental groups have a 
  $\ZZ/p\ZZ$-quotient, see the lists in \cite[\S4-5]{artin}.
  In particular, by what we have just noted, such singularities cannot be linearly reductive quotient singularities.
  Incidentally, the same list of $(\Gamma,p)$ not in the table of Proposition \ref{canonical linearly reductive} 
  gives precisely those canonical surface singularities that are not taut, that is, their analytic isomorphism type is not determined
  by the dual resolution graph of the minimal resolution, see \cite[\S3]{artin}.
  It would be interesting to know whether there is a deeper reason for this coincidence. 
  
  Let us also note that in characteristic $p$, quotients by $\alpha_p$ or $\ZZ/p\ZZ$, both of which are 
  not linearly reductive, may give rise to non-rational singularities, see 
  \cite[Proposition 3.2]{lie uniruled} or \cite{lorenzini}.
\end{Remark}

\subsection{Deformation theory}
\label{subsec:defthy}

Let $X$ be an affine surface over an algebraically closed field $k$
with an isolated singularity.
Let $f:Y\to X$ be the minimal resolution of this singularity,
and $E$ the reduced exceptional divisor.

We start with a result that is implicit in the explicit lists of \cite[\S3]{artin}. 
Although we will not need it in the sequel, we include it
 for completeness of our results and further reference:

\begin{Proposition}[Artin $+\varepsilon$]
 Let $X$ have a canonical singularity that is a linearly
 reductive quotient singularity.
 Let $n$ be the number of $(-2)$-curves of $E$ and $d = \dim_k {\rm Ext}^1(L_{X/k},\OO_X)$.
 \begin{enumerate}
  \item if $X$ has a tame quotient singularity, then $d=n$,
  \item if $X$ has a $D_n$-singularity and $p\geq3$, then $d=n$, and
  \item if $X$ has a wild $A$-singularity, then $d=n+1$.
 \end{enumerate}
 In particular, we find $d\geq n$ in all cases.
\end{Proposition}

\begin{proof}
For any scheme $Z$ of the form $\Spec k[x,y,z]/(g(x,y,z))$, ${\rm Ext}^1(L_{Z/k},\OO_X)$ can be identified with
$
k[x,y,z]/(g,\,\partial_x g,\, \partial_y g,\, \partial_z g),
$
see for example the discussion in \cite[\S4]{CvS}. 
Each of the cases (1)--(3) are of this form; using the explicit equations $g$ obtained in 
\cite[\S3]{artin}, the result follows from straightforward computations.

For example, a singularity of type $A_n$ 
is given in every characteristic by  $g=z^{n+1}-xy$ by \cite[(2.3)]{artin}.
Moreover, $E$ consists of $n$ curves, all of which are $(-2)$-curves.
We compute $\partial_x g=-y$, $\partial_y g=-x$, $\partial_z g=(n+1)z^n$ and find
$$
k[x,y,z]/(g,\,\partial_x g,\, \partial_y g,\, \partial_z g)\,\iso\,
k[z]/(z^{n+1}, (n+1)z^n)\,.
$$
This is a $k$-vector space of dimension $d=n$ if $p$ does not divide $n+1$, that is, if the singularity is tame.
In case, $p$ divides $n+1$, the singularity is wild and then, this vector space is of dimension $d=n+1$.
We leave the remaining cases to the reader.
\end{proof}

The following proposition generalizes \cite[Proposition (1.10)]{bw} to
positive characteristic.
The classical proof over the complex numbers relies on the equivariance 
of these singularities, as well as a vanishing result of 
Tjurina \cite{tjurina}.
Equivariance does not hold for wild $A$-singularities, and
we see in Remark \ref{tjurina fails} 
that Tjurina vanishing fails for every
canonical singularity in every positive characteristic.

\begin{Proposition}[Wahl $+\varepsilon$]
  \label{dimension cohomology}
  Suppose $X$ has a canonical singularity that
  is a linearly reductive quotient singularity.  
  If $n$ denotes the number of $(-2)$-curves of $E$, then
  $$
      n \,=\, \dim_k H^1(Y,\Theta_Y)\,=\,\dim_k H^1_E(Y,\Theta_Y).
  $$ 
\end{Proposition}

\begin{proof}
We have $\dim_k H^1_E(Y,\Theta_Y)=n$ by the proof of \cite[Theorem (6.1)]{wahl}; 
note that the required vanishing results are provided by \cite[Theorem (5.19)]{wahl}.

For the other cohomology group, we consider the short exact sequence
$$
0\,\to\,\Theta_Y(-\log E)\,\to\,\Theta_Y\,\to\,\bigoplus_i N_{E_i}\,\to\,0.
$$
Since $N_{E_i}\iso\OO_{\PP^1}(-2)$ for all $i$, our assertion follows
once we show $H^1(Y, \Theta_Y(-\log E))=0$.
By local duality, we have 
$$
H^1(Y, \, \Theta_Y(-\log E)) \,\iso\, H^1_E(Y, \,  \Theta_Y(-\log E)^\vee\otimes\omega_Y)^\vee,
$$
see \cite[Theorem 4.9]{badescu} for a version that is already adapted to our situation.
Since $Y$ is the minimal resolution of a canonical surface singularity, we have $\omega_{Y/X}\iso\OO_Y$.
Replacing $X$ by an open affine neighborhood of the singularity will not affect
$H^1_E(Y, \,  \Theta_Y(-\log E)^\vee\otimes\omega_Y)$, and we may assume $\omega_Y\iso\OO_Y$.
Moreover, since $\Theta_Y(-\log E)$ is locally free of rank $2$, we compute
$$
  \Theta_Y(-\log E)^\vee\,\iso\, 
   \Theta_Y(-\log E)\,\otimes\,
   \Lambda^2(\Theta_Y(-\log E)^\vee)\,\iso\,\Theta_Y(-\log E)(E),
$$
where the second isomorphism follows from $\omega_Y\iso\OO_Y$ and a local computation (see, for example \cite[Section (1.2)]{wahl quasi}).
By \cite[Theorem (5.19)]{wahl}, $H^1_E(Y,\Theta_Y(-\log E)(E))=0$, and thus,
$H^1(Y,\Theta_Y(-\log E))=0$.
\end{proof}

\begin{Remark}
 Quotients of smooth surfaces by $\mu_n$, which is linearly reductive, 
 give rise to toric singularities, of which $A_n$-singularities are a special case.
 For such singularities, Lee and Nakayama \cite[Proposition 2.11]{lee} established
  the crucial vanishing result  $H^1(Y,\Theta_Y(-\log E))=0$ using 
  toric geometry.
  From this, one can deduce an analogue of Proposition \ref{dimension cohomology} for
  $\mu_n$-quotient singularities.
\end{Remark}

\begin{Remark}[{\bf Failure of Tjurina vanishing}]
 \label{tjurina fails}
  Suppose $X$ has a canonical singularity
  or a rational triple point.
  In characteristic zero, Tjurina \cite{tjurina} proved that 
  $H^1(D,\Theta_D)=0$ for every effective divisor $D$ supported on $E$.
  This vanishing result can be used to prove that these singularities
  are taut,
  see \cite{tjurina} and the discussion at the beginning of \cite[\S2]{laufer}.
  It is also used in the proof of \cite[Proposition (1.10)]{bw},
  which we generalize to positive characteristic in 
  Proposition \ref{prop:deformation canonical singularity} below.

   Let $D$ be an effective divisor supported on $E$ and let ${\mathcal I}_D\subset\OO_Y$ be its
  ideal sheaf.   For every $n\geq1$, the ideal sheaf of $nD$ is ${\mathcal I}_D^n$, and 
  we consider the conormal sequence 
  $$
      {\mathcal I}_D^n/{\mathcal I}_D^{2n}\,\stackrel{\delta}{\longrightarrow}\,\Omega_Y|_{nD}\,\to\,\Omega_{nD}\,\to\,0.
  $$
  For local sections $x$ of ${\mathcal I}_D^n$,  the map $\delta$ is given by $\delta(\overline{x})=dx$.  In particular, if $p$ divides $n$, then $\delta$ is identically zero, and after taking duals, we obtain an isomorphism
  $$
       \Theta_{nD} \,\cong\, \Theta_Y|_{nD}\mbox{ \quad whenever \quad }p|n,
  $$
  which is in stark contrast to characteristic zero, see \cite[(1.6)]{bw}. 
  Next, assume that $-D$ is $f$-ample.
  Replacing $D$ by a sufficiently large multiple, we may assume
  that $H^1(Y,\Theta_Y(-mD))=0$ for all $m\geq1$.
  Taking cohomology in the short exact sequence
  $$
   0\,\to\,\Theta_Y(-nD)\,\to\,\Theta_Y\,\to\,\Theta_Y|_{nD}\,\to\,0
  $$ 
  and using Proposition \ref{dimension cohomology}, we conclude that
  $$
     H^1(nD,\Theta_{nD})\,\neq\,0 \mbox{ \quad whenever \quad }p|n.
  $$
  In particular, Tjurina vanishing fails for every canonical
  singularity in every positive characteristic $p\geq7$.
\end{Remark}

The importance of the cohomology groups $H^1(Y,\Theta_Y)$ and $H^1_E(Y,\Theta_Y)$ considered in Proposition \ref{dimension cohomology} is the following.  The semiuniversal deformation space $\Def_X$ of the singularity $X$ 
has Zariski tangent space ${\rm Ext}^1(L_{X/k},\OO_X)$.
Similarly, the Zariski tangent space of $\Def_Y$ is equal to
$H^1(Y,\Theta_Y)$.
By Proposition \ref{burns-wahl}, we have a morphism
$$
   \Def_Y\,\to\,\Def_X,
$$
which induces a map $\beta:H^1(Y,\Theta_Y)\to{\rm Ext}^1(L_{X/k},\OO_X)$
on Zariski tangent spaces.
Over the complex numbers, $\beta$ is zero \cite[Proposition (1.10)]{bw}; that is, 
a first order deformation of $Y$ induces a locally trivial first-order deformation of $X$.
Moreover, $\ker\beta$ can be identified with the local cohomology group
$H^1_E(\Theta_Y)$, see the short exact sequence on top of \cite[page 73]{bw}. 
In arbitrary characteristic, there is always a map
$\alpha:H^1_E(\Theta_Y)\to\ker\beta$, see, for example, Proposition \ref{prop:deformation canonical singularity} below.
The map $\alpha$ is injective whenever the singularity is equivariant, a property that was studied for canonical singularities
by Wahl \cite[Theorem 5.17]{wahl} and for toric singularities by Lee and Nakayama
\cite[Proposition 2.11]{lee}. For linearly reductive and canonical surface singularities, we have the following result:

\begin{Proposition}
 \label{prop:deformation canonical singularity}
  Suppose $X$ has a canonical singularity that is a linearly reductive quotient singularity. Then there is an exact sequence
  $$
      H^1_E(Y,\Theta_Y)\,\stackrel{\alpha}{\to}\,
      H^1(Y,\Theta_Y)
      \,\stackrel{\beta}{\to}\,
      {\rm Ext}^1(L_{X/k},\OO_X)
  $$
  with $\beta$ as above. Furthermore, 
  \begin{enumerate}
   \item \label{def-can-sing:1} if $X$ is a wild $A$-singularity, then $\alpha$ is not injective
     and $\beta$ is non-zero,
   \item \label{def-can-sing:2} in all other cases, $\alpha$ is an isomorphism and $\beta$ is zero.
  \end{enumerate}
\end{Proposition}

\begin{proof}
We have two exact sequences
\[
\xymatrix{
H^1_E(Y,\Theta_Y)\ar[r]^\alpha & H^1(Y,\Theta_{Y})\ar[r]^\gamma & H^1(Y\backslash E,\Theta_{Y\backslash E})\ar[d]^{\cong}\\
 0\ar[r] & \Ext^1(L_X,\,\OO_X)\ar[r] & H^1(X\backslash{\rm Sing}(X),\, \Theta_{X\backslash{\rm Sing}(X)})
}
\]
where the first row is the long exact sequence of local cohomology,
and the second row is the exact sequence of \cite[Lemma 2]{schlessinger}.
As explained in \cite[(1.15)]{bw}, the image of $\gamma$ lies in
${\rm Ext}^1(L_X,\OO_X)$, and so
$\gamma$ can be identified with the tangent
map $\beta$ to the Burns--Wahl blow-down.

If the singularity is not of type $A_{n-1}$ with $p|n$, 
then $X$ is equivariant, see \cite[Theorem (5.17)]{wahl}.
For such singularities, the map $\alpha$ is injective by 
\cite[Corollary (1.3)]{bw}.
Since $\dim_k H^1_E(\Theta_Y)=\dim_k H^1(\Theta_Y)$ by 
Proposition \ref{dimension cohomology}, 
we find that $\alpha$ is an isomorphism and $\beta$ is zero.

On the hand, if $X$ is an $A_{n-1}$-singularity with $p|n$,
that is, a wild $A$-singularity,
then the torsion sheaf ${\mathcal T}$ in the short exact
sequence
$$
0\,\to\,f_\ast\Theta_Y\,\to\,\Theta_X\,\to\,{\mathcal T}\,\to\,0
$$
is non-trivial, and the inclusion $H^0(\Theta_Y)\to H^0(\Theta_X)$
is strict, see \cite[Remarks (5.18.1)]{wahl}.
From this, we conclude that the restriction map
$H^0(Y,\Theta_Y)\to H^0(Y\smallsetminus E,\Theta_Y)$ is not surjective,
which implies that $\alpha$ is not injective.
Using Proposition \ref{dimension cohomology}, we see that
$\beta$ is non-zero for dimension reasons.
%
\end{proof}

\begin{Remark}
 \label{beta non zero remark}
 In characteristic $2$, the deformation 
 $$
     z^2\,+\,tz\,+\,xy\,=\,0
 $$
 of the $A_1$-singularity over $k[[t]]$ defines a curve inside 
 the semiuniversal deformation space of this singularity.
 This deformation admits a simultaneous resolution of singularities
 over $k[[t]]$, 
 namely by blowing up the ideal $(x,y,z+t)$, see also the discussion 
 in \cite[page 345]{artin simult res}.
 The Burns--Wahl blow-down of this simultaneous resolution gives
 us back the original deformation.
 This shows explicitly that $\beta$ is non-zero.
\end{Remark}

\section{Comparing the minimal resolution and the stacky resolution}
\label{sec:stackyres}

In this section, we prove Theorem \ref{thm:main-stacks}. In \S\ref{subsec:stackyres}, we discuss stacky resolutions of linearly reductive quotient 
singularities and relate lifts of the stack to lifts of its coarse space. 
In \S\ref{subsec:positive-stacks}, we prove (\ref{stacks:X==>Y}) and (\ref{stacks:Y==>X}) of Theorem \ref{thm:main-stacks}, and in \S\ref{subsec:counter-ex-stacks}, 
we prove (\ref{stacks:negative}).

\subsection{Stacky resolutions}
\label{subsec:stackyres}
It is a well-known result that if $X$ is a scheme with tame quotient singularities over a field, then there is a canonical way to endow $X$ with stacky structure in a such way that it becomes smooth.  More precisely, there is a canonical smooth tame Deligne-Mumford stack $\X$ with coarse space $X$ such that the coarse space map $\X\to X$ is an isomorphism over $X^{sm}$ (see \cite[2.9]{int}).

As shown in Proposition \ref{canonical linearly reductive}, most canonical surface singularities are linearly reductive quotient singularities. We are therefore interested in a generalization of the above result for linearly reductive quotient singularities. In this generalization, the role of tame Deligne-Mumford stacks is replaced by the following class of Artin stacks introduced in \cite[Definition 3.1]{aov} (recall our hypotheses from the notation section).

\begin{Definition}
\label{def:tame}
 An Artin stack $\X$ over a base scheme $S$ is called \emph{tame} if the pushforward functor from the category of quasi-coherent sheaves on $\X$ to the category of quasi-coherent sheaves on its coarse space is exact.
\end{Definition}


We then have the following generalization of \cite[2.9]{int}.

\begin{Theorem}[{\cite[Theorem 1.10]{cst}}]
\label{thm:canstack}
  If $X$ is a scheme with linearly reductive quotient singularities over a perfect field $k$, 
  then there is a smooth tame stack $\X$ over $k$ with coarse space $X$.  
  Moreover, if $X^{sm}$ denotes the smooth  locus of $X$, then the induced map
  $$
   \X\,\times_X\,X^{sm}\,\to\, X^{sm}
  $$
  is an isomorphism.
\end{Theorem}

We refer to the above coarse space map $\pi:\X\to X$ as the \emph{stacky resolution} of $X$. 
It is characterized by a universal property (see \cite[Lemma 5.5]{cst}), and so we speak of ``the'' stacky resolution.

\begin{Lemma}
\label{l:affine-stack-defs}
Let $A$ be a complete Noetherian local ring with residue field $k$. Let $G$ be a smooth affine linearly reductive group scheme over $k$ which acts on a smooth affine $k$-scheme $U$, and let $\X=[U/G]$.  Then
\[
\Ext^n(L_{\X/k},\OO_\X)=0
\]
for $n>0$, and so $\X$ has a lift to $A$.
\end{Lemma}
\begin{proof}
Since $\X=[U/G]$, we have a cartesian diagram
\[
\xymatrix{
U\ar[r]\ar[d] & \Spec k\ar[d]^g\\
\X\ar[r]^-{h} & BG
}
\]
This shows that $h$ is smooth and representable. Hence, $L_{\X/BG}$ is a locally free sheaf. From the exact triangle 
\[
g^*L_{BG/k}\to L_{k/k}\to L_{k/BG},
\]
and the fact that $L_{k/k}=0$, we see $g^*L_{BG/k}=L_{k/BG}[-1]$. Since $g$ is smooth and representable, 
$L_{k/BG}$ is isomorphic in the derived category to a locally free sheaf concentrated in degree $0$, 
and so $L_{BG/k}$ is a locally free sheaf concentrated in degree $1$. Using the exact triangle
\[
h^*L_{BG/k}\to L_{\X/k}\to L_{\X/BG},
\]
we see ${\mathcal E}xt^n(L_{\X/k},\OO_\X)=0$ for $n>0$. Since $U$ is affine, $h$ is as well. Since $G$ is linearly reductive, $BG$ is cohomologically affine over $k$, and so composing with $h$, we see that $\X$ is cohomologically affine over $k$, see Definition 3.1 and Proposition 3.9(i) of \cite{gms}. Therefore, for $n>0$, we have
\[
\Ext^n(L_{\X/k},\OO_\X)=H^0({\mathcal E}xt^n(L_{\X/k},\OO_\X))=0,
\]
as desired.
\end{proof}

\begin{Proposition}
\label{prop:stack-def}
Let $X$ and $\X$ be as in Theorem \ref{thm:canstack}, and let $A$ be a complete Noetherian local ring with residue field $k$. 
If $\X'$ is a formal lift of $\X$ to $A$, then $\X'$ has a coarse space $X'$, which is a formal lift of $X$.
\end{Proposition}
\begin{proof}
Since the diagonal of $\X'$ is a deformation of the diagonal of $\X$, it is finite, and so $\X'$ has a coarse space $X'$ by \cite{km}.  Since $\X'$ is flat over $A$, \cite[Corollary 3.3(b)]{aov} shows that $X'$ is as well.  Lastly, \cite[Corollary 3.3(a)]{aov} shows that $X=X'\times_A k$, and so $X'$ is a lift of $X$.
\end{proof}

\subsection{Positive results}
\label{subsec:positive-stacks}

Throughout this subsection, we fix a complete Noetherian local ring $A$ with maximal ideal $\idealm$ and perfect residue field $k$. We fix a surface $X$ over $k$, and let 
$$
f\,:\,Y\,\to\,X
$$
be its minimal resolution of singularities.  We assume that $X$ has 
canonical singularities that are linearly reductive quotient singularities (see Proposition \ref{canonical linearly reductive} and 
Remark \ref{linearly reductive remark} for a complete list of canonical surface singularities with this property). Lastly, we let
$$
\pi\,:\,\X\,\to\,X,
$$
be the stacky resolution of Theorem \ref{thm:canstack}.

We now prove Theorem \ref{thm:main-stacks}(\ref{stacks:X==>Y}).

\begin{Theorem}
\label{thm:X==>YoverA}
If $\X$ lifts formally to $A$, then $Y$ does as well.
\end{Theorem}

\begin{proof}
Let $\X'$ be a lift of $\X$ and let $X'$ be its coarse space. By Proposition \ref{prop:stack-def}, $X'$ is a lift of $X$. We show that the morphism $f:Y\to X$ lifts over $X'$. Since $Y$ is the blow-up of $X$ along a closed scheme $Z$, it is enough to lift $Z$ to a closed subscheme $Z'\subset X'$ with $Z'$ flat over $A$. Indeed, the blow-up of $X'$ along $Z'$ is then flat over $A$ and reduces to $Y$ over $k$.

Note that $Z$ is supported on the singular locus $X^{sing}$ of $X$. Note that 
$X^{sing}$ is a disjoint union of points, as $X$ is normal. Hence,
\[
\Ext^2(L_{Z/X},\OO_X)=\bigoplus_{x\in X^{sing}} {\mathcal E}xt^2(L_{Z/X},\OO_X)_x.
\]
For any \'etale neighborhood $Z_x$ of $Z$ about $x\in X^{sing}$, the obstruction to lifting $Z$ maps to the obstruction to lifting $Z_x$ under the map $\Ext^2(L_{Z/X},\OO_X)\to \Ext^2(L_{Z_x/X},\OO_X)={\mathcal E}xt^2(L_{Z/X},\OO_X)_x$. It therefore suffices to look \'etale locally about each singularity. 
By \cite[Proposition 5.2]{cst}, we can therefore assume that $\X=[U/G]$ with $U$ smooth affine and $G$ a smooth affine linearly reductive group scheme over $k$. By Lemma \ref{l:affine-stack-defs}, there is a lift $\X'$ of $\X$ to $A$.

Let $\Z\subset\X$ be the pullback of $Z\subset X$. It suffices to show that $\Z$ lifts to a closed substack $\Z'$ of $\X'$. Indeed, Lemma 4.14 and Theorem 4.16(ix) of \cite{gms} show that $\Z'$ has a good moduli space $Z'$ which is flat over $A$ with $Z'\subset X'$ a closed subscheme. Proposition 4.7 and Theorem 6.6 of \cite{gms} then show that $Z'$ is a lift of $Z$.

Let $\widetilde{Z}\subset U$ be the pullback of $Z\subset X$. Since $\Z=[\widetilde{Z}/G]$ and $G$ is linearly reductive, the natural map
\[
\Ext^2(L_{\Z/BG},\OO_{\Z})\to \Ext^2(L_{\widetilde{Z}/k},\OO_{\widetilde{Z}})
\]
identifies $\Ext^2(L_{\Z/BG},\OO_{\Z})$ with $\Ext^2(L_{\widetilde{Z}/k},\OO_{\widetilde{Z}})^G$.
Since the obstruction to lifting $\Z$ over $BG'$ maps to the obstruction to lifting $\widetilde{Z}$, it suffices to show that $\widetilde{Z}$ lifts. This follows from \cite[Corollary 8.5]{hartshorne}.
\end{proof}

We now turn to Theorem \ref{thm:main-stacks}(\ref{stacks:Y==>X}).

\begin{Theorem}
If $X$ has no wild $A_n$-singularities and $Y$ lifts formally to $A$, then $\X$ lifts to $A/\idealm^2$.
\end{Theorem}

\begin{proof}
Let $Y'$ be a lift of $Y$ to $A/\idealm^2$, and let $X'$ be the deformation induced by $Y'$, as in Proposition \ref{burns-wahl}. We show that $\X$ lifts over $X'$. Since $\mathcal{E}xt^n(L_{\X/X},\OO_\X)$ is coherent for all $n$, it follows from Definition \ref{def:tame} that 
\[
R\pi_*R\mathcal{H}om(L_{\X/X},\OO_\X)=\pi_*R\mathcal{H}om(L_{\X/X},\OO_\X).
\]
As $\pi$ is an isomorphism over $X^{sm}$, we see that $\pi_*\mathcal{E}xt^n(L_{\X/X},\OO_\X)$ is supported on the singular locus of $X$, which is a disjoint union of points.  As a result, 
\[
\Ext^n(L_{\X/X},\OO_\X)=H^0(\pi_*\mathcal{E}xt^n(L_{\X/X},\OO_\X))=H^0(\mathcal{E}xt^n(L_{\X/X},\OO_\X)).
\]
Hence, the obstruction to lifting $\X$ over $X'$ is a global section of the sheaf $\mathcal{E}xt^2(L_{\X/X},\OO_\X)$. 
To show $\X$ lifts over $X'$, it therefore suffices to look \'etale locally on $X$. By \cite[Proposition 5.2]{cst}, we can assume $\X=[U/G]$ with $U$ affine and $G$ a smooth affine linearly reductive group scheme over $k$. 
Let $\X'$ be a lift of $\X$ over $A$, which exists by Lemma \ref{l:affine-stack-defs}.
By Proposition \ref{prop:stack-def}, the coarse space $X''$ of $\X'$ is a lift of $X$ to $A$. Hence, it suffices to show $X'$ is isomorphic to $X''$.

By the proof of Theorem \ref{thm:X==>YoverA}, there is a lift $Y''$ of $Y$ whose induced deformation of $X$ is $X''$. The map $\beta:H^1(Y,\Theta_Y)\to \Ext^1(L_{X/k},\OO_X)$ sends the class $[Y']-[Y'']$ to $[X']-[X'']$. Since $\beta=0$ by Proposition \ref{prop:deformation canonical singularity}, we see $X'$ is isomorphic to $X''$, as desired.
\end{proof}


\subsection{Counter-examples}
\label{subsec:counter-ex-stacks}

In this subsection, we prove Theorem \ref{thm:main-stacks}(\ref{stacks:negative}), thereby showing that the lifting results of \S\ref{subsec:positive-stacks} are sharp. 

\begin{Theorem}
 \label{non liftable K3}
 Over every algebraically closed field $k$ of characteristic $p=2$, there exists
 \begin{enumerate}
  \item a K3 surface $X$ with (wild) $A_1$-singularities that lifts projectively to $W(k)$ such that
  \item every smooth model of $X$ lifts formally to $W(k)$, whereas
  \item the stacky resolution $\X$ does not lift to $W_2(k)$.
 \end{enumerate}
\end{Theorem}

\begin{proof}
Let $X\to\PP^2_k$ be the purely inseparable double cover defined by
$z^2-f(x_0,x_1,x_2)$, where $f$ is a generic homogeneous polynomial
$f$ of degree $6$.
Then $X$ is a surface with $21$ canonical singularities of type $A_1$,
see \cite[Theorem 3.4]{lie}.
Lifting the double cover over $W(k)$, we conclude that $X$ lifts projectively
to $W(k)$.
The minimal resolution $Y$ of $X$ is a K3 surface.
We have $H^0(Y,\Theta_Y)=0$ by \cite[Theorem 7]{rudakov},
which implies $H^2(Y,\Theta_Y)=0$ by Serre duality using $\omega_Y\iso\OO_Y$,
and so, deformations of $Y$ are unobstructed.
In particular, $Y$ lifts formally to $W(k)$, and thus, 
every smooth model of $X$ lifts formally 
to $W(k)$ by Proposition \ref{smooth surfaces}.

Suppose that $\X$ lifts to $W_2(k)$.  Let $X'$ and $Y'$ be the lifts of $X$ and $Y$ obtained as in the proof of Theorem \ref{thm:X==>YoverA}. With notation as in the proof, $Z$ is the disjoint union of the 21 singular points of $X$ and $Z'\subset X'$ is a closed subscheme which is flat over $W_2(k)$. Since $Y'$ is the blow-up of $X'$ along $Z'$, we see that all 21 exceptional divisors $E_i$ of $Y$ extend to relatively flat Cartier divisors $E'_i$ of $Y'$. 
Since intersection numbers are constant in flat families and for each $i$ we have $E_i^2=-2$, we see $(E'_i)^2=-2$ as well.
We consider the following commutative diagram, whose
downward arrows are restriction maps:
\[
\xymatrix{
  {\rm Pic}(Y')\otimes_\ZZ W_2(k)\ar[r]^-{d\log}\ar[d] & H^1(Y',\Omega^1_{Y'})\ar[d] \\
  {\rm Pic}(Y)\otimes_\ZZ k  \ar[r]^-{d\log} & H^1(Y,\Omega^1_{Y}) 
}
\]
We make the following observations:
\begin{enumerate}
 \item Being a K3 surface, $H^1(Y,\Omega_Y^1)$ is a $k$-vector space of dimension $20$.
  Since the Fr\"olicher spectral sequence from Hodge- to deRham cohomology for $Y$
  degenerates at $E_1$, $H^1(Y,\Omega_Y^1)$ is a subquotient of $\HdR{2}(Y/k)$.
  By semi-continuity, the Fr\"olicher spectral sequence of $Y'$ also degenerates at $E_1$,
  and thus, $H^1(Y',\Omega^1_{Y'})$ is a subquotient of $\HdR{2}(Y'/W_2(k))$.
  Next, the crystalline cohomology of $Y$ is torsion-free, and thus,
  $\HdR{2}(Y'/W_2(k))$ is a free $W_2(k)$-module, and the natural
  reduction map modulo $p$ to $\HdR{2}(Y/k)$ is surjective.
  Putting these observations together, we conclude that 
  $H^1(Y',\Omega^1_{Y'})$ is a free $W_2(k)$-module of rank $20$ and that
  the natural reduction map to $H^1(Y,\Omega_Y^1)$ is surjective.
 \item Serre duality induces a perfect pairing on $H^1(Y,\Omega_Y^1)$ (resp.~ $H^1(Y',\Omega^1_{Y'})$)
   of $k$-modules (resp.~ $W_2(k)$-modules),
 \item the assignment
   $$
    ({\mathcal{L}_1}, {\mathcal{L}_2}) \,\mapsto\, \chi( \mathcal{L}_1^\vee\otimes \mathcal{L}_2^\vee)
     - \chi(\mathcal{L}_1^\vee)-\chi(\mathcal{L}_2^\vee) + \chi(\OO)
   $$
   where $\chi(\mathcal{F}):=\sum_{i=0}^2 (-1)^i\, {\rm length}\, H^i(\mathcal{F})$,
   defines bilinear pairings on ${\rm Pic}(Y)$ and ${\rm Pic}(Y')$, respectively.
   Moreover, the restriction map ${\rm Pic}(Y')\to{\rm Pic}(Y)$ respects the bilinear pairings.
 \item for invertible sheaves $\mathcal{L}_i$, $i=1,2$ on $Y$ (resp.~ $Y'$), we have
   $$
     \left\langle d\log(\mathcal{L}_1),\, d\log(\mathcal{L}_2) \right\rangle_{{\rm Serre\, duality}} 
     \,=\,
     \left\langle \mathcal{L}_1,\,\mathcal{L}_2 \right\rangle_{{\rm Picard\, pairing}} \,\cdot\, 1,
   $$
   see, for example, \cite[Excercise 5.5]{badescu}.
\end{enumerate}

Now, the $E_i'$ are pairwise orthogonal with self-intersection $-2$.
Since $-2\neq0$ in $W_2(k)$, the classes $d\log(E'_i)$ are 
pairwise orthogonal with non-zero self-intersection
(with respect to the pairing coming from Serre duality).
Thus, the classes $d\log(E_1'),...,d\log(E'_{21})$ are linearly independent modulo $2$,
whereas $H^1(Y',\Omega_{Y'}^1)$ is a free $W_2(k)$-module of rank $20$.
This contradiction shows that $\mathcal X$ does not lift to $W_2(k)$.
\end{proof}

\begin{Remark}
 If $\mathcal L$ is a sufficiently ample invertible sheaf on $\PP^2_k$, where $k$ is algebraically closed of positive characteristic $p$, a generic $\alpha_{\mathcal L}$-torsor $X\to\PP^2_k$ will have $A_{p-1}$-singularities only, see \cite[Theorem 3.4]{lie}.  Moreover, $X$ is the canonical model of a surface of general type.
 By lifting the cover, $X$ lifts projectively to $W(k)$. Arguments similar to the ones in the proof of Theorem \ref{non liftable K3} 
 show that the stacky resolution $\X\to X$ does not lift to $W_2(k)$.
 This gives examples in arbitrarily large characteristic of surfaces with
 wild $A_n$-singularities that lift projectively to $W(k)$, but whose
 stacky resolutions do not lift to $W_2(k)$.
\end{Remark}


\begin{thebibliography}{STACKSX}
 \bibitem[AOV08]{aov} D. Abramovich, M. Olsson, and A. Vistoli, \emph{Tame stacks in positive characteristic},
   Ann. Inst. Fourier (Grenoble) 58 (2008), 1057-1091. 
 \bibitem[Al09]{gms} J. Alper, \emph{Good moduli spaces for Artin stacks}, preprint (2009), arXiv:0804.2242v3.
 \bibitem[Ar62]{artin numerical} M. Artin, \emph{Some numerical criteria for contractability of curves on 
  algebraic surfaces}, Amer. J. Math. 84 (1962), 485-496.
 \bibitem[Ar69]{artin approx} M. Artin, \emph{Algebraic approximation of structures over complete local rings},
   Inst. Hautes \'Etudes Sci. Publ. Math. 36 (1969) 23-58. 
 \bibitem[Ar74]{artin simult res} M. Artin, \emph{Algebraic construction of Brieskorn's resolutions},
   J. Algebra 29 (1974), 330-348. 
 \bibitem[Ar77]{artin} M. Artin, {\em Coverings of the rational double points in characteristic $p$},
   Complex analysis and algebraic geometry, 11-22, Iwanami Shoten, 1977.
 \bibitem[Bad01]{badescu} L. B\u{a}descu, \emph{Algebraic surfaces},
   Universitext, Springer-Verlag, New York, 2001.
 \bibitem[BLR90]{BLR} S. Bosch, W. L\"utkebohmert, M. Raynaud, \emph{N\'eron Models}, Ergebnisse der Mathematik
   und ihrer Grenzgebiete 3, Band 21, Springer (1990).
 \bibitem[BHH87]{hirzebruch} G. Barthel, F. Hirzebruch, and T. H\"ofer,
   \emph{Geradenkonfigurationen und Algebraische Fl\"achen},
   Aspects of Mathematics D4, Vieweg 1987.
  \bibitem[Be82]{beauville} A. Beauville, \emph{Les familles stables de courbes elliptiques sur 
    $\PP^1$ admettant quatre fibres singuli\`eres}, C. R. Acad. Sci. Paris S\'er. I Math. 294 (1982), 657-660. 
  \bibitem[BKR01]{mckay1} T. Bridgeland, A. King, and M. Reid, \emph{The {M}c{K}ay correspondence 
  as an equivalence of derived categories}, J. Amer. Math. Soc. 14 (2001), 535--554.   
 \bibitem[BW74]{bw} D. M. Burns and J. M. Wahl, \emph{Local contributions to global deformations of surfaces}, 
   Invent. Math. 26 (1974), 67-88. 
 \bibitem[CR11]{rulling} A. Chatzistamatiou and K. R\"ulling, \emph{Higher direct images of 
    the structure sheaf in positive characteristic}, Algebra and Number Theory 5 (2011), 693-775. 
 \bibitem[CT08]{mckay2} J.--C. Chen and H.--H. Tsend, \emph{A note on derived {M}c{K}ay correspondence}, 
 Math. Res. Lett. 15 (2008), 435--445.    
 \bibitem[CSt09]{CvS} S. Cynk and D. van Straten, \emph{Small resolutions and non-liftable Calabi-Yau threefolds},
   Manuscripta Math.  130  (2009), 233-249.
 \bibitem[DI87]{deligne illusie} P. Deligne and L. Illusie, \emph{Rel\`evements modulo $p^2$ et d\'ecomposition
   du complexe de de Rham}, Invent. Math.  89  (1987), 247-270. 
 \bibitem[Dur79]{durfee} A. H. Durfee, \emph{Fifteen characterizations of rational double points and simple 
   critical points}, Enseign. Math. II. S\'er. 25 (1979), 132-163.
 \bibitem[Ea08]{easton} R. W. Easton, \emph{Surfaces violating Bogomolov-Miyaoka-Yau in positive characteristic},
   Proc. Amer. Math. Soc.  136, 2271-2278 (2008).
 \bibitem[Ek86]{Gauge} T. Ekedahl, \emph{Diagonal complexes and F-gauge structures.},
  Travaux en Cours. Hermann, Paris, 1986.
 \bibitem[EV10]{esnault viehweg} H. Esnault, E. Viehweg, \emph{Surface singularities dominated by smooth 
  varieties}, J. Reine Angew. Math. 649 (2010), 1-9.
  \bibitem[Ha10]{hartshorne} R. Hartshorne, \emph{Deformation theory}, Graduate Texts in Mathematics, 257. Springer, New York, 2010.
 \bibitem[Ill05]{GET} L. Illusie, \emph{Grothendieck's existence theorem in formal geometry},
   Math. Surveys Monogr. 123, Fundamental algebraic geometry, 179-233, AMS (2005). 
 \bibitem[KM97]{km} S. Keel and S. Mori, \emph{Quotients by groupoids}, 
   Ann. of Math. (2) 145 (1997), 193-213.
 \bibitem[La95]{lang} W. E. Lang, \emph{Examples of liftings of surfaces and a problem in de Rham cohomology}, Compositio Math. 97 (1995), 157--160.
 \bibitem[La73]{laufer} H. Laufer, \emph{Taut two-dimensional singularities},
   Math. Ann. 205 (1973), 131-164.
 \bibitem[LN13]{lee} Y. Lee and N. Nakayama, \emph{Simply connected surfaces of general type in positive
   characteristic via deformation theory}, Proc. Lond. Math. Soc. (3) 106 (2013), 225-286.
 \bibitem[Lie08]{lie uniruled} C. Liedtke, \emph{Uniruled surfaces of general type},
   Math. Z. 259 (2008), 775-797. 
 \bibitem[Lie13]{lie} C. Liedtke, \emph{The Canonical Map and Horikawa Surfaces in Positive Characteristic},
   Int. Math. Res. Not. IMRN 2013, no. 2, 422-462.
 \bibitem[Lor]{lorenzini} D. Lorenzini, \emph{Wild quotient singularities of surfaces}, 
   preprint.
 \bibitem[Mu61a]{mumford path} D. Mumford, \emph{Pathologies of modular algebraic surfaces},
   Amer. J. Math. 83 (1961), 339-342. 
 \bibitem[Mu61b]{mumford topology} D. Mumford, \emph{The topology of normal singularities of an algebraic surface and a 
   criterion for simplicity}, Publ. Math. de l�I.H.\'E.S 9 (1961), 5-22.
 \bibitem[Na61]{nagata} M.~Nagata, {\em Complete Reducibility of Rational Representations of a Matrix Group}, 
   J. Math. Kyoto Univ. (1961), 87-99 .
 \bibitem[Ray78]{raynaud} M. Raynaud, \emph{Contre-exemple au ``vanishing theorem'' en caract\'eristique $p>0$},
   C. P. Ramanujam - a tribute, 273-278, Tata Inst. Fund. Res. Studies in Math. 8, Springer (1978). 
 \bibitem[Ray83]{raynaud2} M. Raynaud, \emph{Around the Mordell conjecture for function fields and a conjecture 
   of Serge Lang},  Algebraic geometry (Tokyo/Kyoto, 1982), 1-19,
   Lecture Notes in Math. 1016, Springer 1983.
 \bibitem[RS76]{rudakov} A. Rudakov, I. \v{S}afarevi\v{c}, \emph{Inseparable morphisms of algebraic surfaces},
   Izv. Akad. Nauk SSSR Ser. Mat. 40 (1976), 1269-1307.
 \bibitem[Sa12]{cst} M. Satriano, \emph{The Chevalley-Shephard-Todd theorem for finite linearly reductive 
   group schemes}, Algebra and Number Theory 6--1 (2012), 1--26
 \bibitem[Sch71]{schlessinger} M. Schlessinger, \emph{Rigidity of quotient singularities},
   Invent. Math. 14 (1971), 17-26.
 \bibitem[Sch88]{schoen2} C. Schoen, \emph{On fiber products of rational elliptic surfaces with section},
   Math. Z.  197  (1988), 177-199.
 \bibitem[Sch09]{schoen} C. Schoen, \emph{Desingularized fiber products of semi-stable elliptic surfaces 
   with vanishing third Betti number},  Compos. Math.  145  (2009),  89-111.
 \bibitem[Ser61]{serre} J.-P. Serre, \emph{Exemples de vari\'et\'es projectives en 
   caract\'eristique $p$ non relevables en caract\'eristique z\'ero},
   Proc. Nat. Acad. Sci. 47 (1961), 108-109.
 \bibitem[Se06]{sernesi} E. Sernesi, \emph{Deformations of algebraic schemes},
   Grundlehren der Mathematischen Wissenschaften 334, Springer (2006).
 \bibitem[Suh08]{suh} J. Suh,  \emph{Plurigenera of general type surfaces in mixed characteristic}, 
   Compos. Math. 144, 1214-1226 (2008).
 \bibitem[Tj68]{tjurina} G. Tjurina, \emph{On the tautness of rationally contractible curves on a surface},
   Math. USSR Izvestija 2, 907-934 (1968).
 \bibitem[Vis89]{int} A. Vistoli, \emph{Intersection theory on algebraic stacks and on their moduli spaces}, 
   Invent. Math. 97 (1989), 613-670.
 \bibitem[Wa75]{wahl} J. Wahl, \emph{Vanishing theorems for resolutions of surface singularities},
    Invent. Math. 31 (1975), 17-41.
 \bibitem[Wa85]{wahl quasi} J.~Wahl, {\em A characterization of quasihomogeneous Gorenstein surface singularities},
   Compositio Math. 55, 269-288 (1985).
\end{thebibliography}
\end{document}